\documentclass[amsbook,12pt,a4paper]{article}
\usepackage{amsthm, amsmath, amsfonts, amssymb, xfrac}
\usepackage[usenames]{color}
\usepackage{picinpar}
\usepackage{hyperref}

\usepackage{latexsym}
\usepackage{graphicx}
\usepackage{pgf,pgfarrows,pgfnodes,pgfautomata,pgfheaps,pgfshade}
\newcommand{\nl}{\mbox{}\\}

\newcommand{\ww}{\mbox{\boldmath $w$}}

\newcommand{\uu}{\mbox{\boldmath $u$}}
\newcommand{\zz}{\mbox{\boldmath $z$}}
\newcommand{\RR}{\mathbb{R}}
\newcommand{\PP}{\mathbb{P}}
\newcommand{\HH}{\mbox{\boldmath $H$}}
\newcommand{\norm}[1]{\Vert #1 \Vert}

\newcommand{\intt}{\int_{t_0}^{t} \int_{\RR^3}}

\newcommand{\LL}{{L^2(\RR^3)}}

\newtheorem{lemma}{Lemma}
\begin{document}
\setcounter{page}{1}
%
%
\mbox{} \vspace{-2.000cm} \\
\begin{center}
{\Large \bf
 Asymptotic behavior of global weak solutions } \\
\mbox{} \vspace{-0.300cm} \\
{\Large \bf
for the micropolar dynamics in $L^2(\RR^3)$} \\
%
%
%
\nl
\mbox{} \vspace{-0.300cm} \\
{\sc Robert Guterres , Juliana Nunes and Cilon Perusato } \\
\mbox{} \vspace{-0.350cm} \\
%
%
\nl
\mbox{} \vspace{-0.300cm} \\
%
%
%
%
{\bf Abstract} \\
\mbox{} \vspace{-0.300cm} \\
\begin{minipage}{12.70cm}
{\small
\mbox{} \hspace{+0.250cm}
In this paper the long time behavior of the micropolar fluid equations energy on three dimensional space are studied.  We show that
$ \| (\uu,\ww)(\cdot,t) \|_{{L^2(\RR^3)}} \to 0 $ as $t \to \infty$ for Leray-Hopf's global weak
solutions in inviscid vortex case. 
Moreover, when the vortex viscosity are considered, i.e., $\chi >0$, we obtain a {(\it{faster})} decay for micro-rotational field:   $ \| \ww (\cdot,t) \|_{{L^2(\RR^3)}} = o(t^{-1/2})$. 

}
\end{minipage}
\end{center}

\nl
\mbox{} \vspace{-0.550cm} \\
{\sf AMS Mathematics Subject Classification:}
35Q35 (primary), 35B40, 76D05  \\
\nl
\mbox{} \vspace{-0.800cm} \\
{\sf Key words: } Micropolar fluid equations, Long-time behavior of weak solutions, decay rates in $ L^2 $.  
%
%
%
%

%
%
\nl
\mbox{} \vspace{-0.650cm} \\

{\bf 1. Introduction} \\
\setcounter{section}{1}
\par In this work we derive a long-time asymptotic behavior for Leray global weak solutions of incompressible micropolar equations in $L^2(\RR^3)$, that is, global solutions $(\uu,\ww)(\cdot,t) \in L^{\infty}((0,\infty), \:\!\mbox{\boldmath $L$}^{2}(\mathbb{R}^{n})) 
\,\cap
 $
$ {\displaystyle
	L^{2}((0,\infty), \:\!\dot{\mbox{\boldmath $H$}}\mbox{}^{\!\;\!1}\!\;\!(\mathbb{R}^{n}))
	\cap \;\!
	C_{w}([\;\!0, \infty), \mbox{\boldmath $L$}^{2}(\mathbb{R}^{n}))
} $ of the system 

\begin{subequations}\label{micropolar}
	\begin{equation}\label{uu_equation}
	\mbox{\boldmath $u$}_t
	\;\!+\,
	\mbox{\boldmath $u$} \cdot \nabla \mbox{\boldmath $u$}
	\,+\;\!
	\nabla \:\!{ P}
	\;=\;
	(\:\!\mu + \chi\:\!) \, \Delta \mbox{\boldmath $u$}
	\,+\,
	\chi \, \nabla \times {\bf w},
	\end{equation}
	\begin{equation}\label{ww_equation}
	{\bf w}_t
	\;\!+\,
	\mbox{\boldmath $u$} \cdot \nabla \mbox{\bf w}
	\;=\;
	\gamma \;\! \, \Delta \mbox{\bf w}
	\,+\, \nabla \:\!(\:\!\nabla \cdot\;{\bf w} \:\!)
	\,+\,
	\chi \, \nabla \times\, \mbox{\boldmath $u$}
	\,-\, 2 \, \chi \, {\bf w},
	\end{equation}
	\begin{equation}
	\nabla \cdot \mbox{\boldmath $u$}(\cdot,t) \,=\, 0,
	\end{equation}
\end{subequations}
with initial data 
$ {\displaystyle
	(\:\!\mbox{\boldmath $u$}_0, \mbox{\boldmath $w$}_0) \in
	\mbox{\boldmath $L$}^{2}_{\sigma}(\mathbb{R}^{n})
	\!\times\!
	\mbox{\boldmath $L$}^{2}(\mathbb{R}^{n})
} $, that 
\begin{equation*}
\| (\uu,\ww)(\cdot,t) - (\uu_0,\ww_0) \|_{L^2(\RR^3)}  \to 0,
\end{equation*} 
as $t \to 0$ and such that the \textit{strong} energy inequality\footnote{For the definition of $\|(\uu,\ww) \|_{L^2(\RR^3)} $ , $\|(D\uu,D\ww) \|_{L^2(\RR^3)} $ and other similar expressions throughout the text, see (1.5$e$) and (1.5$f$) }
\begin{equation}\label{energy}
	\begin{split}
		\norm{(\uu,\ww)(\cdot,t)}^2_\LL +2\mu   \int_{t_0}^t \norm{D \uu(\cdot, \tau)}^2_\LL d\tau \\+ 2 \gamma \int_{t_0}^t \norm{D \ww(\cdot, \tau)}^2_\LL d\tau + 2 \int_{t_0}^{t} \norm{\nabla \cdot \ww(\cdot, \tau)}^2_\LL d\tau \\+ 2 \chi \int_{t_0}^{t} \norm{\ww(\cdot,\tau)}^2_\LL d\tau \leq \norm{(\uu,\ww)(\cdot,t_0)}^2_\LL, \, \, \forall t>t_0
	\end{split}
	\end{equation}
for a.e $ t_0 \geq 0$, including $ t_0 = 0$. In (\ref{micropolar}),
$ \mu, \gamma > 0 $ are the kinematic and spin viscosities, and $ \chi \geq 0 $ is the vortex viscosity,  
\mbox{$ \mbox{\boldmath $u$} = \mbox{\boldmath $u$}(x,t) $}, \mbox{$ \mbox{\boldmath $w$} = \mbox{\boldmath $w$}(x,t) $}
 and $ P = P(x,t) $
are the flow velocity, micro-rotational velocity and hydrostatic pressure, respectively, for $ t>0 $ and $ x \in \RR^3 $.
As usual,
$ \!\;\!\mbox{\boldmath $L$}^{2}_{\sigma}(\mathbb{R}^{3}) $
is the
space
of solenoidal fields
$ \:\!\mbox{\bf v} = (v_{1}, v_{2}, v_{3}) \!\:\!\in
\mbox{\boldmath $L$}^{2}(\mathbb{R}^{3}) \equiv L^{2}(\mathbb{R}^{3})^{3} \!\:\!$	
with
$ \nabla \!\cdot \mbox{\bf v} \!\;\!= 0 $
in the distributional sense, 
$ \dot{\mbox{\boldmath $H$}}\mbox{}^{1}(\mathbb{R}^{3}) =
\dot{H}^{1}(\mathbb{R}^{3})^{3} $
where
$ \dot{H}^{1}(\mathbb{R}^{3}) $
denotes the homogeneous Sobolev space of order~$1$,
and
$ C_{w}(I, \:\!\mbox{\boldmath $L$}^{2}(\mathbb{R}^{3})) $
denotes the set of mappings from a
given interval~\mbox{$ I \subseteq \mathbb{R} $}
to
$ \mbox{\boldmath $L$}^{2}(\mathbb{R}^{3}) $
that are $L^{2}$-\;\!weakly continuous
at each $\:\! t \in I$.
By the Leray method to construct weak solutions of problem (\ref{micropolar}) there always exist some $ t_* \gg 1$ - depending on the solution $(\uu,\ww) $ - such that one has 

for each $t_* < T < \infty$, that is, $ (\uu,\ww)(\cdot,t) \in L^\infty_{\text{loc}} ([t_*,\infty), \HH^m(\RR^3))$
\\

\begin{subequations}\label{u_w_suave}
	 \begin{equation}
	 (\uu,\ww) \in C^\infty(\RR^3 \times [t_*,\infty))
	 \end{equation}
	 \mbox{ and, for each $ m \in \mathbb{Z}_+$:}
	 \begin{equation}
	 (\uu,\ww)(\cdot,t) \in L^\infty ([t_*,T), \HH^m(\RR^3)).
	 \end{equation}
\end{subequations}
\par In 1966, A.C. Eringen proposed, in his paper entitled Theory of micropolar fluids
(see \cite{Eringen1966}), a study about the system (\ref{micropolar}). In the literature, such fluids are called micropolar. Physically, micropolar fluids represent fluids consisting of rigid, randomly oriented (or spherical) particles suspended in a viscous medium, where the deformation of fluid particles is ignored. They can describe many phenomena appearing in a large number of complex fluids such as suspensions, blood motion
in animals, and liquid crystals. For more information on these type of fluids, see \cite{Lukaszewicz1999} and
the references therein.

There are many results on the existence and uniqueness of solutions for problems related
to system \ref{micropolar}) (see, for example, \cite{boldrini2010,boldrini1998,cannone2006, chen2012, Eringen1966,galdi1977, Lukaszewicz1999,Rojasmedar1997, yamaguchi2005, Yuan2008, Yuan2010}). Specifically, in 1977, G.P. Galdi and S. Rionero \cite{galdi1977} showed existence and uniqueness of weak solutions to the
initial boundary-value problem for the micropolar system (in this case, $\Omega \subset \RR^3$ is a connected
open set that replaces the whole space $\RR^3$ in (\ref{micropolar}) such that the solution vanishes on  $\partial \Omega \times [0,T]$). For the same problem, G. Lukaszewicz \cite{Lukaszewicz1989} proved existence and uniqueness of strong solutions in 1989, and, in 1990, established the global existence of weak solutions
with sufficiently regular initial data (see \cite{Lukaszewicz1999}). In 1997, M.A. Rojas-Medar \cite{Rojasmedar1997} proved the local existence and uniqueness of strong solutions. E.E. Ortega-Torres and
M.A. Rojas-Medar in 1999, assuming small initial data, proved the global existence of a
strong solution (see \cite{ortega1999}). The results in these last two works were obtained through a spectral
Galerkin method. In 2010, J.L. Boldrini, M. Durán and M.A. Rojas-Medar \cite{boldrini2010} proved
existence and uniqueness of strong solutions in $L^p(\Omega)$, for $p > 3$.

%
%
%
{\bf Notation.}
As shown above,
boldface letters
are used for
vector quantities,
as in
$ {\displaystyle
	\;\!
	\mbox{\boldmath $u$}(x,t)
	=
} $
$ {\displaystyle
	(\:\! u_{\mbox{}_{\!\:\!1}}\!\;\!(x,t), u_{\mbox{}_{\!\:\!2}}\!\;\!(x,t),
	\:\! u_{\mbox{}_{\scriptstyle \!\;\! 3}}\!\;\!(x,t) \:\!)
} $.
$\!$Also,
$ \nabla P \;\!\equiv \nabla P(\cdot,t) $
denotes the spatial gradient of $ \;\!P(\cdot,t) $,
$ D_{\!\;\!j} \!\;\!=\:\! \partial / \partial x_{\!\;\!j} $,
$ {\displaystyle
	\,\!
	\nabla \!\cdot \mbox{\boldmath $u$}
	\:\!=
	D_{\mbox{}_{\!\:\!1}} u_{\mbox{}_{\!\:\!1}} \!\;\!+
	D_{\mbox{}_{\!\:\!2}} u_{\mbox{}_{\!\:\!2}} \!\;\!+
	D_{\mbox{}_{\scriptstyle \!\;\!3}} \,\!
	u_{\mbox{}_{\scriptstyle \!\;\!3}}
} $
is the (spatial) divergence of
$ \:\!\mbox{\boldmath $u$}(\cdot,t) $.
$ |\,\!\cdot\,\!|_{\mbox{}_{2}} \!\,\!\,\!$
denotes the Euclidean norm
in $ \mathbb{R}^{3} \!$,
and
$ {\displaystyle
	\,\!\,\!
	\| \:\!\cdot\:\!
	\|_{\scriptstyle L^{q}(\mathbb{R}^{3})}
	\!\;\!
} $,
$ 1 \leq q \leq \infty $,
are the standard norms
of the Lebesgue spaces
$ L^{q}(\mathbb{R}^{3}) $,
with the vector counterparts \\
\mbox{} \vspace{-0.625cm} \\
\begin{subequations}\label{definitions}
	\begin{equation}
	\|\, \mbox{\boldmath $u$}(\cdot,t) \,
	\|_{\mbox{}_{\scriptstyle L^{q}(\mathbb{R}^{3})}}
	\;\!=\;
	\Bigl\{\,
	\sum_{i\,=\,1}^{3} \int_{\mathbb{R}^{3}} \!
	|\:u_{i}(x,t)\,|^{q} \,dx
	\,\Bigr\}^{\!\!\:\!1/q}
	\end{equation}
	\mbox{} \vspace{-0.700cm} \\
	\begin{equation}
	\|\, D \mbox{\boldmath $u$}(\cdot,t) \,
	\|_{\mbox{}_{\scriptstyle L^{q}(\mathbb{R}^{3})}}
	\;\!=\;
	\Bigl\{\,
	\sum_{i, \,j \,=\,1}^{3} \int_{\mathbb{R}^{3}} \!
	|\, D_{\!\;\!j} \;\!u_{i}(x,t)\,|^{q} \,dx
	\,\Bigr\}^{\!\!\:\!1/q}
	\end{equation}
	\mbox{} \vspace{-0.350cm} \\
	and, in general, \\
	\mbox{} \vspace{-0.750cm} \\
	\begin{equation}
	\|\, D^{m} \mbox{\boldmath $u$}(\cdot,t) \,
	\|_{\mbox{}_{\scriptstyle L^{q}(\mathbb{R}^{3})}}
	\;\!=\;
	\Bigl\{\!\!
	\sum_{\mbox{} \;\;i, \,j_{\mbox{}_{1}} \!,..., \,j_{\mbox{}_{m}} =\,1}^{3}
	\!\;\! \int_{\mathbb{R}^{3}} \!
	|\, D_{\!\;\!j_{\mbox{}_{1}}}
	\!\!\!\;\!\cdot\!\,\!\cdot\!\,\!\cdot \!\:\!
	D_{\!\;\!j_{\mbox{}_{m}}}
	\!\:\! u_{i}(x,t)\,|^{q} \,dx
	\,\Bigr\}^{\!\!\:\!1/q}
	\end{equation}
	\mbox{} \vspace{-0.175cm} \\
	if $ 1 \leq q < \infty $\/;
	if $\, q = \infty $,
	then
	$ {\displaystyle
		\;\!
		\|\, \mbox{\boldmath $u$}(\cdot,t) \,
		\|_{\mbox{}_{\scriptstyle L^{\infty}(\mathbb{R}^{3})}}
		\!=\;\!
		\max \, \bigl\{\,
		\|\,u_{i}(\cdot,t)\,
		\|_{\mbox{}_{\scriptstyle L^{\infty}(\mathbb{R}^{3})}}
		\!\!: \, 1 \leq i \leq 3
		\,\bigr\}
	} $, \linebreak
	\mbox{} \vspace{-0.530cm} \\
	$ {\displaystyle
		\|\, D \,\!\mbox{\boldmath $u$}(\cdot,t) \,
		\|_{\mbox{}_{\scriptstyle L^{\infty}(\mathbb{R}^{3})}}
		\!=\;\!
		\max \, \bigl\{\,
		\|\, D_{\!\;\!j} \;\! u_{i}(\cdot,t)\,
		\|_{\mbox{}_{\scriptstyle L^{\infty}(\mathbb{R}^{3})}}
		\!\!: \:\! 1 \leq i, \:\!j \leq 3
		\,\bigr\}
	} $
	and,
	for general \mbox{$m \!\;\!\geq\!\:\! 1$\/:} \\
	\mbox{} \vspace{-0.550cm} \\
	\begin{equation}
	\|\, D^{m} \mbox{\boldmath $u$}(\cdot,t) \,
	\|_{\mbox{}_{\scriptstyle L^{\infty}(\mathbb{R}^{3})}}
	\!\;\!=\;
	\max\,\Bigl\{\;\!
	\|\;\! D_{\!\;\!j_{\mbox{}_{1}}}
	\!\!\!\;\!\;\!\,\!\cdot \!\;\!\cdot \!\;\!\cdot \!\;\!\;\!
	D_{\!\;\!j_{\mbox{}_{m}}}
	\!\!\;\!\;\! u_{i}(\cdot,t)\,
	\|_{\mbox{}_{\scriptstyle L^{\infty}(\mathbb{R}^{3})}}
	\!\!\!\;\!: \;\!
	1 \leq \!\;\!\;\!
	i, \!\;\!\;\!j_{\mbox{}_{1}}\!\!\;\!\;\!, \!...\!\;\!\;\!,j_{\mbox{}_{m}}
	\!\leq 3
	\!\;\!\;\!\Bigr\}.
	\end{equation}
	\mbox{} \vspace{-0.200cm} \\
	Definitions (\ref{definitions}) are convenient,
	but not essential.
	However,
	some choice for the vector norms
	has to be made to fix the values
	of constants. We defined also for simplicity the following norms for a pair $(\uu,\ww)$ as usually made in literature:
	\begin{equation}
	\| (\uu,\ww) \|^2_{L^q(\RR^3)} := \|\uu\|^2_{L^q(\RR^3)} + \|\ww\|^2_{L^q(\RR^3)}
	\end{equation} 
	and more generally, for all $m \geq 1$ integer
	\begin{equation}
	\| (D^m\uu,D^m\ww) \|^2_{L^q(\RR^3)} := \|D^m\uu\|^2_{L^q(\RR^3)} + \|D^m\ww\|^2_{L^q(\RR^3)}
	\end{equation}
	for all $1 \leq q \leq \infty$. 
\end{subequations}
The constants will be represented by the letters C, c or K. For economy, we will use typically the same symbol to denote constants with different numerical values.  

%
%

\par In \cite{Kato1984}, Kato proved an open problem left by Leray, in 1934. More precisely, he showed that $ \|\uu(\cdot,t)\|_{L^2(\mathbb{R}^{3})} \to 0 $ as $ t \to \infty $ for Navier-Stokes system. Since then, several works have made considerable progress (see e.g \cite{Schonbek}), even for the magnetohydrodynamics case (\cite{SchonbekRubenAgapito2007} and \cite{PerusatoNunesGuterres}). More recently,
in \cite{SchutzZinganoJanainaZingano2015}, the authors showed that
\begin{equation*}
	\lim_{t \to \infty} \norm{\uu(\cdot,t)}_\LL = 0,
\end{equation*}
for Navier-Stokes equations using Duhamel's principle and, with this different approach, other decay problems were solved. In this way, we adapted this technique for the micropolar equations (\ref{micropolar}) and we were able to obtain a faster decay rate for norm $ L^2$ of $\ww$, when $ \chi >0 $, described below. 
\\
\\
\textbf{Main Theorem.}
\\
	\textit{For a Leray solution $(\uu,\ww)(\cdot,t)$ of (\ref{micropolar})  one has}
\begin{subequations}\label{leray}
	\begin{equation}\label{leray_uw}
		\lim_{t \to \infty} \norm{(\uu,\ww)(\cdot,t)}_\LL = 0.
	\end{equation}
	Moreover, if $\chi > 0$ then
	\begin{equation}\label{leray_w}
		\lim_{t \to \infty} t^{1/2} \norm{\ww(\cdot,t)}_\LL = 0.
	\end{equation}
\end{subequations}
However, it was necessary to prove the follow decay property for derivatives
\begin{equation*}
\lim_{t \to \infty} t^{1/2} \norm{(D\uu,D\ww)(\cdot,t)}_\LL = 0
\end{equation*}
and this will be made in the subsequent section. 
\\
\\
\\
{\bf 2. Preliminaries}
\setcounter{section}{2}
\\
\par First, we will obtain the derivatives monotonicity in $L^2(\RR^3)$, for some  $ t_0 >t_*$ large enough. 
\begin{equation*}
\|(D\uu,D\ww)(\cdot,t)\|_{L^2(\RR^3)} \leq \|(D\uu,D\ww)(\cdot,t_0)\|_{L^2(\RR^3)}, \;\;\text{ $ \forall t > t_0 $ }. 
\end{equation*}
This next argument is 
adapted from \cite{KreissHagstromLorenzZingano2003}. Using  (\ref{micropolar}) and  (\ref{u_w_suave}), we get
\begin{equation}\label{D_energia}
	\begin{split}
		\norm{D \uu(\cdot,t)}_\LL^2 + \norm{D \ww(\cdot, t)}_\LL^2 + 2(\mu + \chi)\int_{t_0}^{t} \norm{D^2 \uu(\cdot,\tau)}_\LL^2 d\tau \\ + 2\gamma \int_{t_0}^{t} \norm{D^2 \ww(\cdot,\tau)}_\LL^2 d\tau + 2\int_{t_0}^{t}\norm{D \nabla \cdot \ww(\cdot, \tau)}_\LL^2 d\tau \\ + 4 \chi \int_{t_0}^{t}\norm{D \ww(\cdot,\tau)}_\LL^2 d\tau   \leq \norm{D \uu(\cdot,t_0)}_\LL^2 + \norm{D \ww(\cdot, t_0)}_\LL^2 \\+4\int_{t_0}^{t}\norm{(\uu,\ww)(\cdot,\tau)}_{L^\infty(\mathbb{R}^{3})} \norm{(D\uu,D\ww)(\cdot,\tau)}_\LL \norm{(D^2 \uu,D^2 \ww)(\cdot,\tau)}_\LL d\tau\\ + 4 \chi \intt \sum_{i,j,k,l=1}^{3} \epsilon_{ijk}D_l w_i D_l D_j u_k dx d\tau \leq  
		\norm{(D \uu,D\ww)(\cdot,t_0)}_\LL^2  \\+4\int_{t_0}^{t}\norm{(\uu,\ww)(\cdot,\tau)}^{1/2}_\LL \norm{(D\uu,D\ww)(\cdot,\tau)}^{1/2}_\LL \norm{(D^2 \uu,D^2 \ww)(\cdot,\tau)}^2_\LL d\tau\\ + 4 \chi \int_{t_0}^{t} \bigg{(} \frac{1}{2} \norm{D \ww (\cdot,\tau)}_\LL^2 + \frac{1}{2} \norm{D^2 \uu (\cdot,\tau)}_\LL^2 \bigg{)} d\tau
	\end{split}	
\end{equation}
and so
\begin{equation*}
	\begin{split}
		\norm{(D \uu,D \ww)(\cdot,t)}_\LL^2 + 2\mu\int_{t_0}^{t} \norm{D^2 \uu(\cdot,\tau)}_\LL^2 d\tau  + 2\gamma \int_{t_0}^{t} \norm{D^2 \ww(\cdot,\tau)}_\LL^2 d\tau \\+ 2\int_{t_0}^{t}\norm{D \nabla \cdot \ww(\cdot, \tau)}_\LL^2 d\tau + 2 \chi \int_{t_0}^{t}\norm{D \ww(\cdot,\tau)}_\LL^2 d\tau  \leq \norm{(D \uu,D \ww)(\cdot,t_0)}_\LL^2 \\+ 4\int_{t_0}^{t}\norm{(\uu,\ww)(\cdot,\tau)}^{1/2}_\LL \norm{(D\uu,D\ww)(\cdot,\tau)}^{1/2}_\LL \norm{(D^2 \uu,D^2 \ww)(\cdot,\tau)}^2_\LL d\tau,
	\end{split}
\end{equation*}
where we have used a Sobolev-Nirenberg-Gagliardo (SNG) inequaltie (see (\ref{sobolev})). By (\ref{energy}), we can choose $ t_0 \geq t_*$ large enough such that
\begin{equation*}
C^2 \|(\uu_0,\ww_0)\|_{L^2(\RR^3)} 
\|(D\uu,D\ww)(\cdot,t_0)\|_{L^2(\RR^3)} < (\min \{\mu,\nu \})^2,
\end{equation*}
so that (\ref{D_energia}) gives $ \|(D\uu,D\ww)(\cdot,t)\|_{L^2(\RR^3)} \leq \|(D\uu,D\ww)(\cdot,t_0)\|_{L^2(\RR^3)} $ for all $t$ near $t_0$ by continuity. Actually, with this choice, it follows from ((\ref{D_energia}) again) that  
\begin{equation*} 
C^2 \|(\uu_0,\ww_0)\|_{L^2(\RR^3)} 
\|(D\uu,D\ww)(\cdot,s)\|_{L^2(\RR^3)} < (\min \{\mu,\nu \})^2, \text{       } \forall s \geq t_0.
\end{equation*}
Recalling (\ref{D_energia}), this implies that 
\begin{equation}\label{Du_t<Du_t0}
\|(D\uu,D\ww)(\cdot,t)\|_{L^2(\RR^3)} \leq \|(D\uu,D\ww)(\cdot,t_0)\|_{L^2(\RR^3)}, 
\end{equation}
for all $ t \geq t_0$. Since, by (\ref{energy}), $\| (D\uu,D\ww)(\cdot,t)\|^2_{L^2(\RR^3)} $ is integrable in $(0,\infty)$ one has, by (\ref{Du_t<Du_t0}),  that\footnote{
	Because a monotonic function $ f \in C^0 ((a,\infty)) \cap L^1((a,\infty)) $ has to satisfy $f(t) = o(1/t)$ as $t \to \infty$ (see e.g. \cite{KreissHagstromLorenzZingano2003}, p. 236).  	}  
\begin{equation}\label{Dnorm_to_zero}
\lim_{t \to \infty} t \| (D\uu,D\ww)(\cdot,t)\|^2_{L^2(\RR^3)} = 0.
\end{equation}

\begin{lemma}
	\begin{equation}\label{sobolev}
	\begin{split}
	\|\, (\uu,\ww) \,
	\|_{\mbox{}_{\scriptstyle L^{\infty}(\mathbb{R}^{3})}}
	\,\!
	\|\,  \:\!(D\uu,D\ww) \,
	\|_{\mbox{}_{\scriptstyle L^{2}(\mathbb{R}^{3})}}
	\:\!\\ \leq\; C 
	\|\, (\uu,\ww) \,
	\|_{\mbox{}_{\scriptstyle L^{2}(\mathbb{R}^{3})}}
	^{\:\!1/2}
	\:\!
	\|\,  \:\!(D\uu,D\ww) \,
	\|_{\mbox{}_{\scriptstyle L^{2}(\mathbb{R}^{3})}}
	^{\:\!1/2}
	\:\!
	\|\,  \,\!(D^{2}\uu,D^{2}\ww) \,
	\|_{\mbox{}_{\scriptstyle L^{2}(\mathbb{R}^{3})}}
	\!\:\!,
	\end{split}
	\end{equation}
\end{lemma}

\begin{proof}
	Observe that
	\begin{equation}\label{sobolev0}
		\| \uu \|_{L^q(\RR^3)} \leq \| (\uu,\ww)\|_{L^q(\RR^3)}.
	\end{equation}

	\mbox{} \vspace{-0.650cm} \\
\begin{subequations}
	In $\RR^3$ pointwise values of functions can be
	estimated in terms of $H^{2}\!$ norms. 
	One has \\
	\mbox{} \vspace{-0.650cm} \\
	\mbox{} \vspace{-0.650cm} \\
	\begin{equation}\label{sobolev1}
	\|\, u \,
	\|_{\mbox{}_{\scriptstyle L^{\infty}(\mathbb{R}^{3})}}
	\:\!\leq\;
	\|\, u \,
	\|_{\mbox{}_{\scriptstyle L^{2}(\mathbb{R}^{3})}}
	^{\:\!1/4}
	\:\!
	\|\, D^{2} u \,
	\|_{\mbox{}_{\scriptstyle L^{2}(\mathbb{R}^{3})}}
	^{\:\!3/4}
	\end{equation}
	for $ u \in H^{2}(\mathbb{R}^{3}) $.
	\!These are easily shown
	by Fourier transform and Parseval's identity \linebreak
	(see e.g.\;\cite{SchutzZiebellZinganoZingano2015},
	where the optimal versions of (2.1) and their
	higher dimensional analogues are obtained.
	By Fourier transform,
	we also get
	(for any $n$): \\
	\mbox{} \vspace{-0.650cm} \\
	\begin{equation}\label{sobolev2}
		\|\, D \:\! u  \,
		\|_{\mbox{}_{\scriptstyle L^{2}(\mathbb{R}^{n})}}
		\:\!\leq\;
		\|\, u \,
		\|_{\mbox{}_{\scriptstyle L^{2}(\mathbb{R}^{n})}}
		^{\:\!1/2}
		\,\!
		\|\, D^{2} u \,
		\|_{\mbox{}_{\scriptstyle L^{2}(\mathbb{R}^{n})}}
		^{\:\!1/2}.
	\end{equation}
\end{subequations}

	Combining (\ref{sobolev0}), (\ref{sobolev1}) and (\ref{sobolev2}) we get the desired inequality.
\end{proof}

\begin{lemma}\label{lema2}
	
	Let $ \mbox{u} \in L^{r}(\mathbb{R}^{n}) $ and $e^{\;\!\nu \:\!\Delta \tau}$ the Heat Kernel, then
	\begin{equation}\label{calor-desi-geral}
	\|\, D^{\alpha} \:\![\,
	e^{\;\!\nu \:\!\Delta \tau} \,\!
	\mbox{u} \,]\,
	\|_{\mbox{}_{\scriptstyle L^{2}(\mathbb{R}^{n})}}
	\;\!\leq\;
	K\!\;\!(n,\:\!m)
	\;
	\|\: \mbox{u} \:
	\|_{\mbox{}_{\scriptstyle L^{r}(\mathbb{R}^{n})}}
	\:\!
	(\;\!\nu \;\!\tau\:\!)^{\mbox{}^{\scriptstyle
			\!\! -\, \frac{\scriptstyle n}{2}
			\left( \frac{1}{\scriptstyle r} \;\!-\;\! \frac{1}{2} \right)
			\,-\, \frac{|\:\!{\scriptstyle \alpha}\:\!|}{2} }}
	%
	\end{equation}
	\mbox{} \vspace{-0.125cm} \\
	for all $ \tau > 0 $
	and
	$ \alpha $ (multi-index), 		
	$ 1 \leq r \leq 2 $,
		$ n \geq 1 $,
	and
	$ \;\!m = |\;\!\alpha\:\!| $.
	($\:\!$For a proof of (\ref{calor-desi-geral}),
	see e.g.$\;$\cite{KreissLorenz1989,  			
		LorenzZingano2012})
\end{lemma}
\,
\,
\,
\,
\,
\,
\


%
%
{\bf 3. Proof of Main Theorem }
\\
\par First we will prove the result for $\ww$. By (\ref{Dnorm_to_zero}), given $\epsilon >0$, there exists $t_0 > 0$ such that
\begin{equation}
	\norm{(D \uu, D\ww)(\cdot,t)}_\LL< \epsilon t^{-\frac{1}{2}} \mbox{ \,\,\,\, } \forall t>t_0.
\end{equation}

We begin with the inviscid vortex case, i.e., $\chi = 0$. Rewrite equation (\ref{ww_equation}) as
\begin{equation*}
	\ww_t = \gamma \Delta \ww + Q_1,
\end{equation*}
where  $Q_1 = -\uu \cdot \nabla \ww + \nabla(\nabla \cdot \ww)$. By Duhamel's Principle, we get
\begin{equation}
	\begin{split}
		\norm{\ww(\cdot,t)}_\LL \leq \underbrace{\norm{e^{\gamma \Delta(t-t_0)}\ww(\cdot,t_0)}_\LL}_{I} + \underbrace{\int_{t_0}^{t} \norm{e^{\gamma \Delta(t-s)}  \uu \cdot \nabla \ww}_\LL ds}_{II} \\ + \underbrace{\int_{t_0}^{t} \norm{e^{\gamma \Delta(t-s)} \nabla (\nabla \cdot \ww)}_\LL ds}_{III}.
	\end{split}
\end{equation}  

$I$ is the solution of heat equation with initial condition $\ww(\cdot,t_0) \in L^2(\RR^3)$, and so, 
\begin{equation}\label{heatkernel_to_zero}
	\lim_{t \to \infty} \norm{e^{\gamma \Delta(t-t_0)}\ww(\cdot,t_0)}_\LL = 0.
\end{equation}
To estimate the other terms, we use lemma (\ref{lema2}). 
\begin{equation*}
	\begin{split}
	II=	\int_{t_0}^{t} \norm{e^{\gamma \Delta(t-s)}  \uu \cdot \nabla \ww}_\LL ds \leq K \gamma^{-3/4} \int_{t_0}^{t} (t-s)^{-3/4} \norm{\uu \cdot \nabla \ww}_{L^1(\RR^{3})} ds \\ \leq K \gamma^{-3/4}\int_{t_0}^{t} (t-s)^{-3/4} \norm{\uu}_\LL \norm{D \ww}_\LL ds \\ \leq K \epsilon \gamma^{-3/4} \norm{(\uu,\ww)(\cdot,t_0)}_\LL \int_{t_0}^{t} (t-s)^{-3/4} s^{-1/2} ds
	\\ \leq 
	\int_{t_0}^{t} (t-s)^{-3/4} s^{-1/2} ds \leq C t^{-1/4}.
\end{split}
\end{equation*}
Therefore,
\begin{equation}\label{uu_ww_norm_t1e3}
	\lim_{t \rightarrow \infty}\int_{t_0}^{t} \norm{e^{\gamma \Delta(t-s)}  \uu \cdot \nabla \ww}_\LL ds = 0.
\end{equation}
Similarly,  
\begin{equation*}
	\begin{split}
		III = \int_{t_0}^{t} \norm{e^{\gamma \Delta(t-s)} \nabla (\nabla \cdot \ww)}_\LL ds \leq K \gamma^{-1/2} \int_{t_0}^{t} (t-s)^{-1/2} \norm{D \ww}_\LL ds \\ \leq K \epsilon \gamma^{-1/2} \int_{t_0}^{t} (t-s)^{-1/2} s^{-1/2} ds
		\\ \leq
		C \epsilon.
	\end{split}
\end{equation*}
Therefore
\begin{equation}\label{uu_ww_norm_t1e4}
	\lim_{t \rightarrow \infty} \int_{t_0}^{t} \norm{e^{\gamma \Delta(t-s)} \nabla (\nabla \cdot \ww)}_\LL ds = 0.
\end{equation}
That is, if $\chi =0$, then

\begin{equation*}
	\lim_{t \rightarrow \infty} \norm{\ww(\cdot,t)}_\LL = 0.
\end{equation*}

Now, suppose that $\chi > 0$. Rewrite equation (\ref{ww_equation}) as

\begin{equation*}
\ww_t = \gamma \Delta \ww + Q_1 - 2 \chi \ww,
\end{equation*}
where  $Q_1 = -\uu \cdot \nabla \ww + \nabla(\nabla \cdot \ww)$. Define $\zz = e^{2 \chi t} \ww$. In this way, we obtain that
\begin{equation*}
	\zz_t = \gamma \Delta \zz + e^{2 \chi t} Q_1.
\end{equation*}
Applying Duhamel's Principle, we have
\begin{equation*}
\zz(\cdot , t) = e^{\gamma \Delta(t - t_0)}\zz(\cdot , t_0) + \int_{t_0}^{t}e^{\gamma \Delta(t - s)}e^{2 \chi s} Q_1(\cdot , s) ds,
\end{equation*}
that is,
\begin{equation*}
	\begin{split}
		\ww(\cdot , t) = e^{- 2 \chi(t - t_0)} e^{\gamma \Delta(t - t_0)}\ww(\cdot , t_0) - \int_{t_0}^{t} e^{-2 \chi (t - s)}e^{\gamma \Delta(t - s)}\uu \cdot \nabla \ww ds \\ + \int_{t_0}^{t} e^{-2 \chi (t - s)}e^{\gamma \Delta(t - s)}\nabla(\nabla \cdot \ww) ds + \chi \int_{t_0}^{t} e^{-2 \chi (t - s)}e^{\gamma \Delta(t - s)}(\nabla \times \uu) ds.
	\end{split}
\end{equation*}
Multiplying by $t^{1/2}$, applying $L^{2}$ norm  and Minkowski inequality, we obtain
\begin{equation*}
	\begin{split}
		t^\frac{1}{2} \norm{\ww(\cdot , t)}_\LL \leq \underbrace{t^\frac{1}{2} e^{- 2 \chi(t - t_0)}\norm{e^{\gamma \Delta(t - t_0)}\ww(\cdot , t_0)}_\LL}_{I}\\ 
		+ \underbrace{t^\frac{1}{2}\int_{t_0}^{t} e^{-2 \chi (t - s)}\norm{e^{\gamma \Delta(t - s)}\uu \cdot \nabla \ww}_\LL ds}_{II} \\
		+ \underbrace{t^\frac{1}{2} \int_{t_0}^{t} e^{-2 \chi (t - s)}\norm{e^{\gamma \Delta(t - s)}\nabla(\nabla \cdot \ww)}_\LL ds}_{III} \\
		+ \underbrace{\chi t^\frac{1}{2} \int_{t_0}^{t} e^{-2 \chi (t - s)}\norm{e^{\gamma \Delta(t - s)}\nabla \times \uu}_\LL ds}_{IV}.
	\end{split}
\end{equation*}
By (\ref{heatkernel_to_zero}), we have that 
\begin{equation*}
   \lim_{t \to \infty}t^\frac{1}{2} e^{- 2 \chi(t - t_0)}\norm{e^{\gamma \Delta(t - t_0)}\ww(\cdot , t_0)}_\LL = 0.
\end{equation*}
To estimate the other terms, we use again lemma $2$.
\begin{equation*}
	\begin{split}
	 II = t^{1/2}\int_{t_0}^{t} e^{-2 \chi (t - s)}\norm{e^{\gamma \Delta(t - s)}\uu \cdot \nabla \ww}_\LL ds \\ \leq K \gamma^{-3/4} t^{1/2}\int_{t_0}^{t} e^{-2 \chi (t - s)} (t-s)^{-3/4} \norm{\uu \cdot \nabla \ww}_{L^1(\RR^3)}ds \\ \leq K\gamma^{-3/4} t^{1/2}\int_{t_0}^{t} e^{-2 \chi (t - s)} (t-s)^{-3/4} \norm{\uu}_\LL \norm{D \ww}_\LL ds \\ \leq  K\gamma^{-3/4} \norm{(\uu,\ww)(\cdot,t_0)}_\LL t^{1/2}\int_{t_0}^{t} e^{-2 \chi (t - s)} (t-s)^{-3/4} \norm{D \ww}_\LL ds \\ \leq K \epsilon \gamma^{-3/4} \norm{(\uu,\ww)(\cdot,t_0)}_\LL t^{1/2}\int_{t_0}^{t} e^{-2 \chi (t - s)} (t-s)^{-3/4} s^{-1/2} ds
	 \\\leq K \epsilon \gamma^{-3/4} \norm{(\uu,\ww)(\cdot,t_0)}_\LL\Bigg( e^{- \chi t}t^{1/4} + (2 \chi)^{-1/4} \Gamma \bigg{(}\frac{1}{4} \bigg{)}\Bigg).
	\end{split}
\end{equation*}
Therefore
\begin{equation*}
\lim_{t \rightarrow \infty}t^{1/2}\int_{t_0}^{t} e^{-2 \chi (t - s)}\norm{e^{\gamma \Delta(t - s)}\uu \cdot \nabla \ww}_\LL ds = 0.
\end{equation*}
Similarly,
\begin{equation*}
	\begin{split}
		III = t^{1/2} \int_{t_0}^{t} e^{-2 \chi (t - s)}\norm{e^{\gamma \Delta(t - s)}\nabla(\nabla \cdot \ww)}_\LL ds \\ \leq K \gamma^{-1/2} t^{1/2}  \int_{t_0}^{t} e^{-2 \chi (t - s)} (t-s)^{-1/2} \norm{D \ww}_\LL ds \\ \leq K \epsilon \gamma^{-1/2} t^{1/2}  \int_{t_0}^{t} e^{-2 \chi (t - s)} (t-s)^{-1/2} s^{-1/2} ds
	\\ \leq K \epsilon \gamma^{-1/2}  t^{1/2} \int_{t_0}^{t} e^{-2 \chi (t - s)} (t-s)^{-1/2} s^{-1/2} ds \leq K \epsilon \gamma^{-1/2}  \Bigg( e^{- \chi t}t^{1/2} + (2 \chi)^{-1/2} \sqrt{\pi}\Bigg).
	\end{split}
\end{equation*}
Therefore
\begin{equation*}
	\lim_{t \rightarrow \infty}t^\frac{1}{2} \int_{t_0}^{t} e^{-2 \chi (t - s)}\norm{e^{\gamma \Delta(t - s)}\nabla(\nabla \cdot \ww)}_\LL ds = 0.
\end{equation*}
Finally,
\begin{equation*}
	\begin{split}
		\chi t^{1/2} \int_{t_0}^{t} e^{-2 \chi (t - s)}\norm{e^{\gamma \Delta(t - s)}\nabla \times \uu}_\LL ds \leq \chi t^{1/2} \int_{t_0}^{\frac{t}{2}} e^{-2 \chi (t - s)}\norm{e^{\gamma \Delta(t - s)}\nabla \times \uu}_\LL ds \\+ \chi t^{1/2} \int_{\frac{t}{2}}^{t} e^{-2 \chi (t - s)}\norm{e^{\gamma \Delta(t - s)}\nabla \times \uu}_\LL ds 
\end{split}
\end{equation*}
\begin{equation*} 
\begin{split} 		
		 \leq K \chi \gamma^{-1/2} t^{1/2} e^{-\chi t} \int_{t_0}^{\frac{t}{2}} (t-s)^{-1/2} \norm{\uu}_\LL ds + \chi \epsilon t^{1/2} \int_{\frac{t}{2}}^{t} e^{-2\chi(t-s)} s^{-1/2} ds  \\\leq K \bigg{[} \chi \gamma^{-1/2}  \norm{(\uu, \ww)(\cdot,t_0)}_\LL t^{1/2} e^{-\chi t} \int_{t_0}^{\frac{t}{2}} (t-s)^{-1/2} ds \\+ \chi \epsilon t^{1/2}  t^{-1/2} \int_{\frac{t}{2}}^{t} e^{-2\chi(t-s)} ds \bigg{]}  \leq K \bigg{[} \chi \gamma^{-1/2}  \norm{(\uu, \ww)(\cdot,t_0)}_\LL e^{-\chi t} + \chi \epsilon \frac{1 - e^{-\chi t}}{2 \chi} \bigg{]}.
	\end{split}
\end{equation*}
Therefore
\begin{equation}\label{ww_norm_e4}
	\lim_{t \rightarrow \infty}\chi t^{1/2} \int_{t_0}^{t} e^{-2 \chi (t - s)}\norm{e^{\gamma \Delta(t - s)}\nabla \times \uu}_\LL ds = 0.
\end{equation}
That is, if $\chi > 0$, then
\begin{equation*}
\lim_{t \rightarrow \infty} t^{1/2}\norm{\ww(\cdot,t)}_\LL = 0.
\end{equation*}
Now, we show that $\norm{\uu(\cdot,t)}_\LL \to 0$ whether $\chi = 0$ or $\chi > 0$. Rewrite equation (\ref{uu_equation}) as
\begin{equation*}
	\uu_t = (\mu + \chi) \Delta \uu + Q_2.
\end{equation*}

Where $Q_2 = \chi \nabla \times \ww - \uu \cdot \nabla \uu - \nabla p$. We can write $Q_2$ as

\begin{equation*}
	Q_2 = \PP_h [\chi \nabla \times \ww - \uu \cdot \nabla \uu],
\end{equation*}
where $\PP_h$ denotes the Helmholtz-Leray decomposition. By Duhamel's Principle
\begin{equation*}
	\uu(\cdot, t) = e^{(\mu + \chi)\Delta(t-t_0)} \uu(\cdot, t_0) + \int_{t_0}^t e^{(\mu + \chi)\Delta(t-s)} Q_2(\cdot, s) ds.
\end{equation*}

Since heat kernel commutes with Helmholtz Projector, applying  $L^2$ norm and Minkowski's inequality, we have

\begin{equation*}
	\begin{split}
		\norm{\uu(\cdot, t)}_\LL \leq \underbrace{\norm{e^{(\mu + \chi)\Delta(t-t_0)} \uu(\cdot, t_0)}_\LL}_{I} + \underbrace{\int_{t_0}^t \norm{e^{(\mu + \chi)\Delta(t-s)} \uu \cdot \nabla \uu}_\LL ds}_{II}\\ + \underbrace{\chi \int_{t_0}^t \norm{e^{(\mu + \chi)\Delta(t-s)} \nabla \times \ww}_\LL ds}_{III}.
	\end{split}
\end{equation*}
As in (\ref{heatkernel_to_zero}), $I$ is the solution of the heat equation with initial condition $\uu(\cdot,t_0)$. Hence,
\begin{equation*}
	\lim_{t \to \infty} \norm{e^{\gamma \Delta(t-t_0)}\uu(\cdot,t_0)}_\LL = 0.
\end{equation*}
To estimate the other terms, we use again lemma $2$.
\begin{equation*}
	\begin{split}
		\int_{t_0}^t \norm{e^{(\mu + \chi)\Delta(t-s)} \uu \cdot \nabla \uu}_\LL ds \leq K (\mu + \chi)^{-\frac{3}{4}} \int_{t_0}^{t} (t-s)^{-\frac{3}{4}} \norm{\uu \cdot \nabla \uu}_{L^1(\RR^3)} ds \\ \leq K (\mu + \chi)^{-\frac{3}{4}} \int_{t_0}^{t} (t-s)^{-\frac{3}{4}} \norm{\uu(\cdot,s)}_\LL \norm{D \uu (\cdot, s)}_\LL ds \\ \leq K (\mu + \chi)^{-\frac{3}{4}} \norm{(\uu,\ww)(\cdot,t_0)}_\LL \int_{t_0}^{t} (t-s)^{-\frac{3}{4}} \norm{D \uu (\cdot, s)}_\LL ds \\
		< K \epsilon \int_{t_0}^{t} (t-s)^{-\frac{3}{4}} s^{-\frac{1}{2}} ds \leq K \epsilon t^{-\frac{1}{4}}.
	\end{split}
\end{equation*}
Therefore
\begin{equation*}
\lim_{t \rightarrow \infty} \int_{t_0}^t \norm{e^{(\mu + \chi)\Delta(t-s)} \uu \cdot \nabla \uu}_\LL ds = 0.
\end{equation*}
We only need to worry about $III$ when $\chi > 0$, but in this case we may assume that $t_0$ is big enough such that
\begin{equation*}
\norm{\ww(\cdot,t)}_\LL < \epsilon t^{-\frac{1}{2}} \mbox{ } \forall t>t_0
\end{equation*}	
and so, we have
\begin{equation*}
	\begin{split}
		\chi \int_{t_0}^t \norm{e^{(\mu + \chi)\Delta(t-s)} \nabla \times \ww}_\LL ds	\leq K \chi (\mu + \chi)^{-1/2} \int_{t_0}^t (t-s)^{-1/2} \norm{\ww}_\LL ds \\ \leq K \chi \epsilon (\mu + \chi)^{-1/2} \int_{t_0}^t (t-s)^{-1/2} s^{-1/2} ds \\ \leq K \chi \epsilon (\mu + \chi)^{-1/2}.
	\end{split}
\end{equation*}
Therefore
\begin{equation}\label{uu_ww_norm_e4}
	\lim_{t \rightarrow \infty} \chi \int_{t_0}^t \norm{e^{(\mu + \chi)\Delta(t-s)} \nabla \times \ww}_\LL ds = 0.
\end{equation}
That is,
\begin{equation*}
	\lim_{t \to \infty} \norm{\uu(\cdot,t)}_\LL = 0
\end{equation*}
and the proof of main theorem is complete.

%
%
\nl
\mbox{} \vspace{-1.750cm} \\
%

%

%
%
%

%
%

\nl
\mbox{} \vspace{-0.250cm} \\
\nl
{\small

\nl
\mbox{} \vspace{-0.450cm} \\
\nl
\begin{minipage}[t]{10.00cm}
	\mbox{\normalsize \textsc{Cilon Perusato}} \\
	Departamento de Matem\'atica Pura e Aplicada \\
	Universidade Federal do Rio Grande do Sul \\
	Porto Alegre, RS 91509-900, Brazil \\
	E-mail: {\sf cilonperusato@gmail.com} \\
	
	\mbox{\normalsize \textsc{Juliana Nunes}} \\
	Departamento de Matem\'atica Pura e Aplicada \\
	Universidade Federal do Rio Grande do Sul \\
	Porto Alegre, RS 91509-900, Brazil \\
	E-mail: {\sf juliana.s.ricardo@gmail.com} \\
	
	\mbox{\normalsize \textsc{Robert Guterres}} \\
	Departamento de Matem\'atica Pura e Aplicada \\
	Universidade Federal do Rio Grande do Sul \\
	Porto Alegre, RS 91509-900, Brazil \\
	E-mail: {\sf rguterres.mat@gmail.com} \\
\end{minipage}
}
%
%

%
%

\end{document}